\documentclass{article}
\usepackage{amsthm,color,amsmath,graphicx,mathtools,amssymb,tikz}
\usepackage{graphv1}
\usepackage[T1]{fontenc}
\usetikzlibrary{calc,decorations.pathmorphing,decorations.pathreplacing,shapes,arrows.meta,positioning,fit}

\newcommand{\com}[1]{}
\newtheorem{lemma}{Lemma}
\newtheorem{theorem}{Theorem}
\newtheorem{corollary}{Corollary}
\newtheorem{observation}{Observation}
\newtheorem{definition}{Definition}
\newenvironment{question}[1][]{\medskip\par\noindent
\textbf{Question\if\relax\detokenize{#1}\relax\else~(#1)\fi. }}{\medskip\par}

\newcommand\step{%
\mathrel{%
\begin{tikzpicture}[baseline= {( $ (current bounding box.south) + (0,-.4ex) $ )}]
  \path[draw,{Classical TikZ Rightarrow}-{Classical TikZ Rightarrow}] (0,0) -- (3ex,0);
\end{tikzpicture}}%
}

\newcommand\reach[2][]{%
\mathrel{%
\begin{tikzpicture}[baseline= {( $ (a.south) + (0,-.75ex) $ )}]
  \node[inner sep=.5ex, minimum width=3ex ] (a) {$\scriptscriptstyle #2$};
  \node[below=0.1ex of a, inner sep=.5ex] (b) {$\scriptscriptstyle #1$};
  \node[fit=(a)(b), inner sep=0pt] (tempbox) {};
  \draw[{Classical TikZ Rightarrow}-{Classical TikZ Rightarrow}, decorate,
    decoration={zigzag, amplitude=0.7pt, segment length=1.2mm, pre=lineto, pre length=3pt, post=lineto, post length=3pt}] (tempbox.west |- a.south) -- (tempbox.east |- a.south);
\end{tikzpicture}}%
}

\newcommand\nreach[2][]{%
\mathrel{%
\begin{tikzpicture}[baseline= {( $ (a.south) + (0,-.75ex) $ )}]
  \node[inner sep=.5ex, minimum width=3ex ] (a) {$\scriptstyle #2$};
  \path[draw,{Classical TikZ Rightarrow}-{Classical TikZ Rightarrow},decorate,
    decoration={zigzag,amplitude=0.7pt,segment length=1.2mm,pre=lineto, pre length=3pt, post=lineto, post length=3pt}] (a.south east) -- (a.south west);
 \node[below=0.1ex of a, inner sep=.5ex] {$\scriptscriptstyle #1$};
 \path[draw] ( $ (a.south) + (0.5ex,1ex) $ ) -- ( $ (a.south) + (-0.5ex,-1ex) $ );
\end{tikzpicture}}%
}
\title{A linear upper bound on the number of moves required for independent set reconfiguration with two sliding tokens\medskip}
\author{Nived J. M.\\\small Faculty of Mathematics, Institute of Algebra,\\\small
TU Dresden, Germany\\\small email: \texttt{nived.jooly\_manojan@tu-dresden.de}\medskip
\and
Mathew C. Francis\\\small Department of Computer Science and Engineering,\\\small Indian Institute of Technology Palakkad, India\\\small email: \texttt{mathew@iitpkd.ac.in}}
\date{}

\begin{document}
\maketitle
\begin{abstract}
    We consider the problem of shifting two tokens placed on nonadjacent vertices $u,v$ of a graph $G$ on $n$ vertices to two nonadjacent vertices $u',v'$ of $G$ using a sequence of token movements. In each step, a token is moved from the vertex it is on to a neighbour of that vertex, ensuring that the tokens remain on nonadjacent vertices after this move. We answer a question of Bria\'nski, Felsner, Hodor, and Micek [``Reconfiguring Independent Sets on Interval Graphs'', \textit{MFCS 2021}] by showing that if the two tokens can be moved from their initial position to their final position, then it can be done using at most $4n$ moves.
\end{abstract}
\section{Introduction}
Given an optimization problem on graphs, the reconfiguration version of that problem asks whether one feasible solution (for the given optimization problem) can be transformed into another feasible solution by following a sequence of steps, each of which has to obey certain specified rules. In most cases, we want each intermediate step in this transformation also to result in a feasible solution. In this work, we consider the \emph{independent set reconfiguration} problem: given two independent sets $I$ and $I'$ in a graph, we study the number of steps required in transforming $I$ to $I'$ following certain rules. Specifically, we study independent set reconfiguration under the \emph{token sliding} rule. In this paradigm, an independent set in the host graph is seen as a ``configuration'' of tokens placed on the vertices in the independent set. The tokens are initially placed on the vertices of $I$. In each step, a token can be moved from the vertex it is on to a neighbour of that vertex, so that after the move, the vertices occupied by tokens form an independent set of the graph. Note that if an independent set $I$ can be transformed into an independent set $I'$ under the token sliding rule, then $|I|=|I'|$. The central question that we study is the minimum number of moves required to transform one independent set into another.

Another way to view the independent set reconfiguration problem is using \emph{reconfiguration graphs}. Given a graph $G$, define the reconfiguration graph $\mathcal{R}_k(G)$ to be a graph with with vertex set $\{I\colon I$ is an independent set of size $k$ in $G\}$ and edge set $\{II'\colon I$ can be transformed to $I'$ in one step$\}$. Note that the edge set of $\mathcal{R}_k(G)$ will depend upon the transformation rules under consideration (token sliding in our case). 
Also note that the minimum number of moves required in transforming an independent set $I$ of cardinality $k$ to an independent set $I'$ is exactly the distance between $I$ and $I'$ in $\mathcal{R}_k(G)$, which we denote by $\mathrm{dist}_{\mathcal{R}_k(G)}(I,I')$ (it is possible that $I$ cannot be transformed into $I'$; this happens when $I$ and $I'$ lie in different connected components of $\mathcal{R}_k(G)$, and therefore $\mathrm{dist}_{\mathcal{R}_k(G)}(I,I')=\infty$). Let $\mu_k(n)$ denote the maximum number of moves required to transform any independent set $I$ of size $k$ in any graph $G$ on $n$ vertices to another independent set $I'$ in $G$, provided that $I$ can be transformed into $I'$. In other words, $$\mu_k(n)=\max_{\substack{G\\|V(G)|=n}}\quad\max_{\substack{I,I'\in V(\mathcal{R}_k(G))\\\mathrm{dist}_{\mathcal{R}_k(G)}(I,I')<\infty}} \mathrm{dist}_{\mathcal{R}_k(G)}(I,I')$$

Since for any graph $G$ on $n$ vertices, we have $|V(\mathcal{R}_k(G))|\leq\genfrac{(}{)}{0pt}{}{n}{k}=O(n^k)$, we know that $\mu_k(n)$ is upper bounded by $O(n^k)$. 

(\textit{Note on asymptotic notation:} Given two functions $f:\mathbb{N}\rightarrow\mathbb{N}$ and $g:\mathbb{N}\rightarrow\mathbb{N}$, we say that ``$f(n)$ is upper bounded by $O(g(n))$'' or ``$f(n)\leq O(g(n))$'' to mean that there exists a function $g'(n)\in O(g(n))$ such that $f(n)\leq g'(n)$ for all $n\in\mathbb{N}$.)

Bria\'nski, Felsner, Hodor and Micek~\cite{brianskietal} observe that there is no known example that shows a superlinear lower bound on $\mu_k(n)$ for any value of $k\geq 2$. (Clearly, $\mu_1(n)\leq n-1$ since a token on any vertex can be moved to any other vertex in at most $n-1$ steps, provided there is a path between the two vertices.)

\begin{question}[\cite{brianskietal}]
Is $\mu_k(n)\leq O(n)$ for every $k\in\mathbb{N}$?
\end{question}

In fact, they note that the question of whether there is a linear upper bound for $\mu_k(n)$ is open even for $k=2$.
We answer this question in the affirmative by proving the following theorem.

\begin{theorem}\label{thm:main}
    $\mu_2(n)\leq 4n$.
\end{theorem}

How tight is Theorem~\ref{thm:main}? Clearly, for a graph $G$ that is isomorphic to a path on $n$ vertices having vertex set $\{v_1,v_2,\ldots,v_n\}$ and edge set $\{v_iv_{i+1}\colon 1\leq i\leq n-1\}$, where $n\geq 4$, any transformation from the independent set $\{v_1,v_3\}$ to the independent set $\{v_{n-2},v_n\}$ using token sliding will require at least $2n-6$ moves. So $\mu_2(n)\geq 2n - 6$.

We view the transformation of an independent set $I$ of cardinality 2 to another independent set $I'$ under token sliding as the process of moving two tokens, initially placed on the vertices of $I$, along the edges of the graph so that they are finally on the vertices of $I'$, in such a way that the two tokens are never on the two endpoints of some edge in any intermediate configuration. We say that a vertex $v$ is ``visited'' by a token when a token moves to $v$ from a neighbouring vertex, or if the token starts on that vertex. Thus, in the example of the path graph above, when the tokens move from $\{v_1,v_3\}$ to $\{v_{n-2},v_n\}$ using $2n-6$ moves, each vertex in $\{v_1,v_2,v_{n-1},v_n\}$ gets visited once (by some token) and every other vertex gets visited twice.

\begin{figure}
    \centering
    \begin{tikzpicture}
    \renewcommand{\vertexset}{(a,0,0),(b,1,0,black),(c,2,0,lightgray),(d,3,0),(e,4,0,lightgray),(f,2,1,black)}
    \renewcommand{\edgeset}{(a,b),(b,c),(c,d),(d,e),(c,f)}
    \renewcommand{\defradius}{0.1}
    \drawgraph
    \node [below = 1mm] at (\xy{b}) {$x$};
    \end{tikzpicture}
    \caption{Moving the two tokens on the black vertices to the two gray vertices will require a token being on the vertex $x$ twice.}
    \label{fig:twice}
\end{figure}
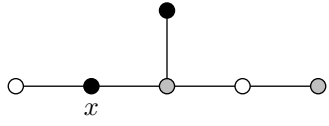

A natural question then would be whether it is always possible to move the tokens on $I$ to $I'$ in such a way that neither of the two tokens ever visits a vertex twice (but a vertex visited by one token can be visited later by the other token). A moment's thought convinces us that the answer is negative; a simple example is shown in Figure~\ref{fig:twice}. The example of the path shows that the number of vertices that are required to be visited more than once by tokens could be large. But could it be the case that there is always a way to move the tokens from $I$ to $I'$ such that the number of vertices that are visited more than once by \emph{an individual token} is small? After all, in the example of the path graph, neither of the tokens individually needs to visit any vertex more than once, and in the example shown in Figure~\ref{fig:twice}, there is only one vertex that gets visited more than once by an individual token. But as Figure~\ref{fig:manytimes} shows, one can construct examples in which any sequence of token movements transforming $I$ to $I'$ has to involve an individual token visiting a large number of vertices more than once. Note that we do not know of an example in which \emph{both} the tokens need to visit a large number of vertices more than once. 

\begin{figure}
    \centering
    \begin{tikzpicture}
    \renewcommand{\vertexset}{(y1,3,0,lightgray),(b,2,1,black),(a,1,1,black),(p,1,2),(z,3,2),(u,3,3,lightgray),(y2,3.5,0),(y3,4,0),(y4,5.5,0),(y5,6,0),(q,6.5,0)}
    \renewcommand{\edgeset}{(y1,a),(a,z),(z,b),(b,y1),(a,p),(p,z),(z,u),(y1,y2),(y2,y3),(y3,y4,,,,dotted),(y4,y5),(y5,q),(z,y1),(z,y2),(z,y3),(z,y4),(z,y5)}
    \renewcommand{\defradius}{0.1}
    \drawgraph
    \node at (3.9,0.9) {$\cdots$};
    \node [right=1mm] at (\xy{b}) {$x$};
    \draw [decorate, decoration={brace, amplitude=10pt, mirror}, thick] (2.9,-0.2) -- (6.1,-0.2);
    \node at (4.5,-0.75) {$S$};
    \end{tikzpicture}
    \caption{Moving the tokens on the two black vertices to the two gray vertices will require the token at $x$ visiting every vertex in $S$ at least twice.}
    \label{fig:manytimes}
\end{figure}
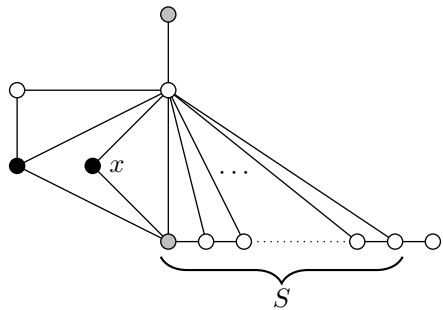

Suppose that we call the two vertices in $I$ the ``source vertices'' and the two vertices in $I'$ the ``target vertices''.
The discussion above tells us that if we want to move \emph{both} the tokens from the source vertices to the target vertices, some vertices might need to get visited more than once. We show in Theorem~\ref{thm:final} that no sequence of token movements that can move the tokens from the source vertices to the target vertices ever needs any vertex to get visited by tokens more than four times, which is what leads to the proof of Theorem~\ref{thm:main}. But what if we insist that no vertex should be visited more than once? We show in Corollary~\ref{cor:abtouy} that if we wanted to place a token only on one of the two target vertices, then this can always be achieved in a way such that no vertex gets visited by tokens more than once.

\section{Preliminaries}
\subsection{Basic definitions and notation}
We always assume that $G=(V,E)$ is an finite, undirected, simple graph. We follow Diestel~\cite{diestel} for all standard graph-theoretic notation and terminology. A set $S\subseteq V(G)$ is an \emph{independent set} in $G$ if no two vertices in $S$ are adjacent in $G$. A graph $H$ is a \emph{subgraph} of $G$ if $V(H)\subseteq V(G)$ and $E(H)\subseteq E(G)$. A subgraph $H$ of $G$ is an \emph{induced subgraph} of $G$ if for each $u,v\in V(H)$, we have that $uv\in E(H)$ if and only if $uv\in E(G)$. The induced subgraph of $G$ having vertex $S\subseteq V(G)$, sometimes referred to as ``the subgraph induced in $G$ by $S$'', is denoted as $G[S]$; it is the graph having vertex set $S$ and edge set $\{uv\colon u,v\in S$ and $uv\in E(G)\}$.
For $S\subseteq V(G)$, we denote by $G-S$ the graph $G[V(G)\setminus S]$.
A \emph{connected component} of a graph $G$ is a maximal connected subgraph of $G$. 
For a connected component $C$ of a graph, we shall use $C$ and $V(C)$ interchangeably, for the sake of brevity of notation.

If $P$ is a path in $G$ and $u,v\in V(P)$, then we denote by $uPv$ the subpath of $P$ between $u$ and $v$. If $P_1$ and $P_2$ are two paths such that $V(P_1)\cap V(P_2)$ contains exactly one vertex, and that vertex is an end-vertex of both $P_1$ and $P_2$, then we denote by $P_1P_2$ the path obtained by combining the two paths $P_1$ and $P_2$. Thus given two paths $P,Q$ in a graph $G$ and $x,y,z,w\in V(G)$, $xyPwQz$ is the path obtained by combining the edge $xy$, the subpath of $P$ between $y$ and $w$, and the subpath of $Q$ between $w$ and $z$ (to use this notation, it has to be ensured that $x$ does not lie on $yPw$ or $wQz$, and that the only common vertex for the paths $yPw$ and $wQz$ is $w$). Please refer to Diestel~\cite{diestel} for a more detailed explanation of this notation.
\subsection{Reconfiguration sequences}
For two independent sets $I, I'$ of a graph $G$, we say that $I\step I'$ if $I\setminus I'=\{a\}$, $I'\setminus I=\{b\}$, and $ab\in E(G)$. Thus $I\step I'$ if and only if $I$ can be transformed in to $I'$ (and vice versa) in one step under the token sliding rule. We sometimes describe this one-step transformation from $I$ to $I'$ using the phrase ``move the token on $a$ along the edge $ab$ to $b$''. Let $\reach{}$ denote the reflexive and transitive closure of $\step$. Thus, $I\reach{\quad} I'$ if and only if there exist independent sets $I=I_1,I_2,\ldots,I_k=I'$ such that for each $i\in\{2,3,\ldots,k\}$, we have $I_{i-1}\step I_i$. We call the sequence $I_1,I_2,\ldots,I_k$ a \emph{reconfiguration sequence} from $I$ to $I'$ of \emph{length} $k$. Notice that since $\step$ is symmetric, there exists a reconfiguration sequence from $I$ to $I'$ if and only if there exists a reconfiguration sequence from $I'$ to $I$. It follows that $\reach{}$ is an equivalence relation on the set of independent sets of $G$. For $X\subseteq V(G)$, we denote by $I\reach{X} I'$ the fact that there exists a reconfiguration sequence $I=I_1,I_2,\ldots,I_k=I'$ such that $I_1\cup I_2\cup\cdots\cup I_k\subseteq X$. Note that $\reach{X}$ is also an equivalence relation.

\begin{definition}\label{def:reachable}
Let $I$ be an independent set in $G$. We say that a vertex $u\in V(G)$ is \emph{reachable} from $I$ if there exists an independent set $I'$ of $G$ containing $u$ such that $I\reach{} I'$.
\end{definition}

For independent sets $I, I'$ of $G$, we say that $I\reach[t]{} I'$, for some positive integer $t$, if there exists a sequence of independent sets $I=I_1,I_2,\ldots,I_k=I'$ such that for each $i\in\{2,3,\ldots,k\}$, we have $I_{i-1}\step I_i$, and for each $u\in V(G)$, $\sum_{i\in\{1,2,\ldots,k\}} |(I_i\setminus I_{i-1})\cap\{u\}|\leq t$, where we assume $I_0=\emptyset$. Further, if $X\subseteq V(G)$ such that $I_1\cup I_2\cup\cdots\cup I_k\subseteq X$, then we say that $I\reach[t]{X} I'$. It is not difficult to see that for any two independent sets $I,I'$ of $G$, $X\subseteq V(G)$, and any positive integer $t$, $I\reach[t]{X} I'$ if and only if $I'\reach[t]{X} I$.

The following observation is also easy to see.
\begin{observation}\label{obs:length}
If $I\reach[t]{X} I'$, there is a reconfiguration sequence of length at most $|X|t$ between $I$ and $I'$.    
\end{observation}

Clearly, $I\reach{} I'$ and $I\reach[t]{} I'$ are equivalent to $I\reach{V(G)} I'$ and $I\reach[t]{V(G)} I'$ respectively. The observation below follows directly from
the definitions.

\begin{observation}\label{obs:concat}
If $I\reach[i]{X} I'$ and $I'\reach[j]{Y} I''$, then $I\reach[i+j]{X\cup Y} I''$. Further, if $X\cap Y\subseteq I'$, then $I\reach[i+j-1]{X\cup Y} I''$.    
\end{observation}

\begin{observation}\label{obs:easymove}
Let $\{a,b\}$ and $\{a,c\}$ be independent sets of $G$ and $X\subseteq V(G)$ such that $a,b,c\in X$. If there is a path in $G[X]$ between $b$ and $c$ that does not contain any neighbour of $a$, then $\{a,b\}\reach[1]{X}\{a,c\}$.
\end{observation}
\begin{proof}
Let $b=x_1,x_2,\ldots,x_k=c$ be a path in $G[X]$ between $b$ and $c$ that does not contain any neighbour of $a$. Since the vertices $x_1,x_2,\ldots,x_k$ are pairwise distinct vertices in $X$, and each of them is distinct from $a$, we have $\{a,b\}=\{a,x_1\}\step\{a,x_2\}\step\cdots\step\{a,x_k\}=\{a,c\}$, and further that $\{a,b\}\reach[1]{X}\{a,c\}$.
\end{proof}

\section{Proof of Theorem~\ref{thm:main}}

\begin{definition}
For $S\subseteq V(G)$ let $G_S$ denote the graph $G-S$, and for any vertex $u\in V(G_S)$, let $C_S(u)$ denote the connected component of $G_S$ that contains $u$. For a vertex $u\in S$, we denote by $R_S(u)$ the connected component of $G_{S\setminus\{u\}}$ that contains $u$, i.e. $R_S(u)=C_{S\setminus\{u\}}(u)$. For a vertex $a\in V(G)$, we abbreviate $G_{N(a)}$ and $C_{N(a)}(u)$ to just $G_a$ and $C_a(u)$ respectively. Also, for vertices $a,b\in V(G)$, we abbreviate $G_{N(a)\cap N(b)}$, $C_{N(a)\cap N(b)}(u)$, and $R_{N(a)\cap N(b)}(u)$  to just $G_{ab}$, $C_{ab}(u)$, and $R_{ab}(u)$ respectively.
\end{definition}

\begin{definition}
Let $S\subseteq V(G)$. For a connected component $C$ of $G_S$ and a vertex $v\in S$, we say that:
\vspace{-0.05in}
\begin{itemize}
    \renewcommand{\itemsep}{0in}
    \item[-] $v$ is \emph{adjacent to} $C$ if there exists $w\in C$ such that $vw\in E(G)$,
    \item[-] $v$ is \emph{nonadjacent to} $C$ if there does not exist any vertex $w\in C$ such that $vw\in E(G)$,
    \item[-] $v$ is \emph{avoided by} $C$ if there exists $w\in C$ such that $vw\notin E(G)$, and
    \item[-] $v$ is \emph{dominated by} $C$ if there does not exist any vertex $w\in C$ such that $vw\notin E(G)$.
\end{itemize}
\end{definition}

Observe that ``$v$ is not adjacent to $C$'' is equivalent to ``$v$ is nonadjacent to $C$'', and ``$v$ is not dominated by $C$'' is equivalent to ``$v$ is avoided by $C$''.

Notice that for $S\subseteq V(G)$ and vertices $u,v\in V(G)\setminus S$, we have $C_S(u)=C_S(v)$ if and only if there is a path in $G$ between $u$ and $v$ that does not contain any vertex in $S$.
Also observe that for a vertex $u\in S$, $R_S(u)$ is the subgraph induced in $G$ by $u$ together with the vertices in each connected component of $G_S$ that is adjacent to $u$.
Another fact worth noting is that if $u\in C_S(v)$, then $v\in C_S(u)$.

\begin{lemma}\label{lem:impossible}
Let $S\subseteq V(G)$ and let $C_1, C_2$ be connected components of $G_S$, where possibly $C_1=C_2$. Let $v_1\in C_1$ and $v_2\in C_2$. If both $C_1$ and $C_2$ dominate every vertex in $S$, then for every independent set $I$ such that $\{v_1,v_2\}\reach{} I$, it must be the case that $I=\{x_1,x_2\}$ where $x_1\in C_1$ and $x_2\in C_2$.
\end{lemma}
\begin{proof}
Let us define a ``bad'' independent set to be an independent set $I$ that is reachable from $\{v_1,v_2\}$ having the property that $\nexists x_1\in C_1, x_2\in C_2$ such that $I=\{x_1,x_2\}$. The lemma states that if both $C_1$ and $C_2$ dominate every vertex in $S$, there are no bad independent sets.
Suppose for the sake of contradiction that there is at least one bad independent set. Let $I$ be a bad independent set such that the length of a shortest possible reconfiguration sequence from $\{v_1,v_2\}$ to $I$ is as small as possible. Then the independent set $I'$ that occurs just before $I$ in a shortest length reconfiguration sequence from $\{v_1,v_2\}$ to $I$ is not bad (note that $I'$ exists since $\{v_1,v_2\}$ is not bad). That means that $I'=\{x'_1,x'_2\}$, where $x'_1\in C_1$ and $x'_2\in C_2$.
As $I'\step I$, we clearly have that either $x'_1\in I$ or $x'_2\in I$. Let us assume without loss of generality that $x'_1\in I$. Then $I=\{x'_1,y\}$, where $x'_2y\in E(G)$. As $I$ is bad, we know that $y\notin C_2$. As $x'_2\in C_2$ and $x'_2y\in E(G)$, it follows that $y\in S$. But since $C_1$ dominates every vertex in $S$, we have that $x'_1y\in E(G)$, which contradicts the fact that $I$ is an independent set.
\end{proof}
\begin{corollary}\label{cor:impossible}
Let $\{a,b\}$ be an independent set in $G$ and let $u\in V(G_{ab})\setminus C_{ab}(b)$. If every vertex in $N(a)\cap N(b)$ is dominated by $C_{ab}(u)$ and $C_{ab}(a)$, then $\{a,b\}\nreach{}\{a,u\}$.
\end{corollary}
\begin{lemma}\label{lem:firstmove}
    Let $\{a,b\}$ be an independent set in $G$ and $w\in C_{ab}(a)\cup C_{ab}(b)$. Then either $w\in C_a(b)$ or $w\in C_b(a)$, i.e. there exists a path from one of $a,b$ to $w$ that does not contain any neighbour of the other.
\end{lemma}
\begin{proof}
Since $\{a,b\}$ is an independent set, the statement of the lemma is trivially true if $w\in\{a,b\}$. So let us assume that $w\notin\{a,b\}$.
As $w\in C_{ab}(a)\cup C_{ab}(b)$, there is a path in $G_{ab}$ between $w$ and one of $a$ or $b$. Among all such paths, let $P$ be one of minimum length. Without loss of generality, assume that $P$ is a path between $a$ and $w$, since the other case is symmetric. We have $b\notin V(P)$; otherwise, the path $bPw$ betwen $b$ and $w$ is shorter than $P$, thereby contradicting the choice of $P$. We claim that $P$ contains no vertex of $N(b)$. Suppose, for a contradiction, that there exists a vertex $x\in V(P)\cap N(b)$. Clearly, $x\neq a$ since $\{a,b\}$ is an independent set. If $x\in N(a)$, then $x\in N(a)\cap N(b)$, contradicting the fact that $V(P)\subseteq V(G_{ab})$. Hence $x\notin N(a)\cup\{a\}$. Since $x$ is adjacent to $b$ but not to $a$, the path $bxPw$ is strictly shorter than $P$, again contradicting the choice of $P$. Therefore, $V(P)\cap N(b)=\emptyset$. Thus $P$ does not contain any neighbour of $b$ (which implies that $P$ also does not contain $b$), and hence $P$ is also a path in $G_b$. Consequently, $w\in C_b(a)$.
\end{proof}
\begin{corollary}\label{cor:firstmove}
Let $\{a,b\}$ be an independent set in $G$ and let $w\in C_{ab}(a)\cup C_{ab}(b)=:X$. Then either $\{a,b\}\reach[1]{X}\{a,w\}$ or $\{a,b\}\reach[1]{X}\{b,w\}$.
\end{corollary}
\begin{proof}
Notice that the path between one of $\{a,b\}$ to $w$ which does not contain any neighbour of the other that is guaranteed by Lemma~\ref{lem:firstmove} is a path in $G_{ab}$ and therefore also a path in $G[C_{ab}(a)\cup C_{ab}(b)]=G[X]$. The statement of the corollary now follows from Observation~\ref{obs:easymove}.
\end{proof}
\begin{definition}
Let $\{a,b\}$ be an independent set in $G$. For $u\in V(G)\setminus (C_{ab}(a)\cup C_{ab}(b))$, we say that $(v_0,v_1,\ldots,v_k)$ is an ``$\{a,b\}$-chain to $u$'' having length $k$ if:
\vspace{-0.05in}
\begin{itemize}
    \renewcommand{\itemsep}{0in}
    \item[-] for each $i\in\{0,1,\ldots,k\}$, $v_i\in N(a)\cap N(b)$,
    \item[-] for each $i\in\{1,2,\ldots,k\}$, we have either $v_{i-1}v_i\notin E(G)$ or there exists a connected component of $G_{ab}$ that is adjacent to $v_{i-1}$ and avoids $v_i$,
    \item[-] $v_0$ is avoided by $C_{ab}(a)$ or $C_{ab}(b)$, and
    \item[-] if $u\in N(a)\cap N(b)$, then $v_k=u$; otherwise, $v_k$ is adjacent to $C_{ab}(u)$.
\end{itemize}
\end{definition}
\begin{observation}\label{obs:minimumchain}
    Let $\{a,b\}$ be an independent set in $G$ and let $u\in V(G)\setminus (C_{ab}(a)\cup C_{ab}(b))$. Let $(v_0,v_1,\ldots,v_k)$ be a minimum length $\{a,b\}$-chain to $u$. Then:
    \vspace{-0.05in}
    \begin{enumerate}
        \setlength{\itemsep}{0in}
        \renewcommand{\theenumi}{(\roman{enumi})}
        \renewcommand{\labelenumi}{\theenumi}
        \item\label{distinct} $v_0,v_1,\ldots,v_k$ are pairwise distinct,
        \item\label{init} $C_{ab}(a)$ and $C_{ab}(b)$ dominate every vertex in $\{v_1,v_2,\ldots,v_k\}$,
        \item\label{ucomponent} if $u\notin N(a)\cap N(b)$, then no vertex in $\{v_0,v_1,\ldots,v_{k-1}\}$ is adjacent to $C_{ab}(u)$, and
        \item\label{nonconsec} for all $i,j\in\{0,1,\ldots,k\}$ such that $j>i+1$,
        \vspace{-0.05in}
        \begin{enumerate}
            \renewcommand{\theenumii}{(\alph{enumii})}
            \renewcommand{\labelenumii}{\theenumii}
            \item\label{ijadjacent} $v_iv_j\in E(G)$,
            \item\label{jdomination} $v_j$ is dominated by every component of $G_{ab}$ that is adjacent to $v_i$, and
        \end{enumerate}
    \end{enumerate}
\end{observation}
\begin{proof}
\ref{distinct} If $v_i=v_j$, where $0\leq i<j\leq k$, then $(v_0,v_1,\ldots,v_{i-1},v_j,v_{j+1},\ldots,v_k)$ is an $\{a,b\}$-chain to $u$ that has length smaller than $k$, contradicting the fact that $(v_0,v_1,\ldots,v_k)$ is a minimum length $\{a,b\}$-chain to $u$.
\medskip

\noindent\ref{init} If one of $C_{ab}(a)$ or $C_{ab}(b)$ does not dominate $v_i$, for some $i\in\{1,2,\ldots,k\}$, then it means that one of $C_{ab}(a)$ or $C_{ab}(b)$ avoids $v_i$. Then $(v_i,v_{i+1},\ldots,v_k)$ is an $\{a,b\}$-chain to $u$ that has length smaller than $k$, again giving the same contradiction as above.
\medskip

\noindent\ref{ucomponent} Suppose that $u\notin N(a)\cap N(b)$ and there exists $i\in\{0,1,\ldots,k-1\}$ such that $v_i$ is adjacent to $C_{ab}(u)$. Then again, we have a similar contradiction as $(v_0,v_1,\ldots,v_i)$ is an $\{a,b\}$-chain to $u$ that has length smaller than $k$.
\medskip

\noindent\ref{nonconsec} Suppose that $0\leq i<i+1<j\leq k$.
If $v_iv_j\notin E(G)$, then $(v_0,v_1,\ldots,v_i,v_j,v_{j+1},\ldots,v_k)$ is an $\{a,b\}$-chain to $u$ that has length smaller than $k$. This proves~\ref{ijadjacent}. If $v_j$ is not dominated by some component $C$ of $G_{ab}$ that is adjacent to $v_i$, then $C$ is a component of $G_{ab}$ that avoids $v_j$ and is adjacent to $v_i$. Therefore $(v_0,v_1,\ldots,v_i,v_j,v_{j+1},\ldots,v_k)$ is an $\{a,b\}$-chain to $u$ that has length smaller than $k$. This proves~\ref{jdomination}.
\end{proof}

\begin{lemma}\label{lem:reconfig}
Let $\{a,b\}$ be an independent set in $G$, $u\in V(G)\setminus (C_{ab}(a)\cup C_{ab}(b))$, and $(v_0,v_1,\ldots,v_k)$ a minimum length $\{a,b\}$-chain to $u$. Let $X_0=C_{ab}(a)\cup C_{ab}(b)$ and for each $i\in\{1,2,\ldots,k\}$, let $X_i=R_{ab}(v_{i-1})$ and $Z_i=\{v_i\}\cup\bigcup_{0\leq j\leq i} X_j$. Then for each $i\in\{0,1,\ldots,k\}$, there exists $y_i\in X_i$ such that $\{a,b\}\reach[1]{Z_i}\{v_i,y_i\}$.
\end{lemma}
\begin{proof}
We shall prove this by induction on $i$. As the base case for the induction, consider the case when $i=0$. Recall that $v_0$ is avoided by $C_{ab}(a)$ or $C_{ab}(b)$, which implies that there exists $w\in C_{ab}(a)\cup C_{ab}(b)$ such that $wv_0\notin E(G)$. Since $v_0\in N(a)\cap N(b)$ and we have by Corollary~\ref{cor:firstmove} that $\{a,b\}\reach[1]{X_0}\{a,w\}$ or $\{a,b\}\reach[1]{X_0}\{b,w\}$, either $\{a,b\}\reach[1]{X_0}\{a,w\}\step\{v_0,w\}$ or $\{a,b\}\reach[1]{X_0}\{b,w\}\step\{v_0,w\}$. As $v_0\notin C_{ab}(a)\cup C_{ab}(b)=X_0$, we then have $\{a,b\}\reach[1]{Z_0}\{v_0,w\}$ and we are done by setting $y_0=w$. For the inductive step, assume that $i>0$ and that there exists $y_{i-1}\in X_{i-1}$ such that $\{a,b\}\reach[1]{Z_{i-1}}\{v_{i-1},y_{i-1}\}$. Notice that $y_{i-1}$ is a vertex in $R_{ab}(v_{i-2})$ if $i>1$, or a vertex in $C_{ab}(a)\cup C_{ab}(b)$ otherwise. In the former case, we have by Observation~\ref{obs:minimumchain}\ref{ijadjacent} and~\ref{obs:minimumchain}\ref{jdomination} that $y_{i-1}v_i\in E(G)$. In the latter case, by Observation~\ref{obs:minimumchain}\ref{init}, we again have $y_{i-1}v_i\in E(G)$. Note that we have by Observation~\ref{obs:minimumchain}\ref{distinct} that $v_{i-1}\neq v_i$. By the definition of a chain, we have that either $v_{i-1}v_i\notin E(G)$, or there exists a connected component $C'$ of $G_{ab}$ that is adjacent to $v_{i-1}$ and avoids $v_i$. In the former case, we have $\{a,b\}\reach[1]{Z_{i-1}}\{v_{i-1},y_{i-1}\}\step\{v_{i-1},v_i\}$, and since $v_i\notin Z_{i-1}$ and $Z_{i-1}\subseteq Z_i$, it follows that $\{a,b\}\reach[1]{Z_i}\{v_{i-1},v_i\}$. We are done by setting $y_i=v_{i-1}$. In the latter case, since $C'$ avoids $v_i$, we have a vertex $w\in C'$ such that $v_iw\notin E(G)$. We have by Observation~\ref{obs:minimumchain}\ref{jdomination} that no vertex in $\{v_0,v_1,\ldots,v_{i-2}\}$ is adjacent to $C'$, and we have by Observation~\ref{obs:minimumchain}\ref{init} that $C'\neq C_{ab}(a)$ and $C'\neq C_{ab}(b)$. This implies that no vertex of $C'$ is adjacent to any vertex of $X_{i-1}$. As $y_{i-1}\in X_{i-1}$, it follows that $y_{i-1}$ is not adjacent to any vertex in $C'$.
Thus $y_{i-1}$ does not have any neighbour on an arbitrarily chosen path between $v_{i-1}$ and $w$ in $G[\{v_{i-1}\}\cup C']$. We now have from Observation~\ref{obs:easymove} that $\{v_{i-1},y_{i-1}\}\reach[1]{C'\cup\{v_{i-1},y_{i-1}\}}\{w,y_{i-1}\}$.
As $Z_{i-1}\cap (C'\cup\{v_{i-1},y_{i-1}\})=\{v_{i-1},y_{i-1}\}$ (as no vertex in $\{v_0,v_1,\ldots,v_{i-2}\}$ is adjacent to $C'$) and $\{a,b\}\reach[1]{Z_{i-1}}\{v_{i-1},y_{i-1}\}$, we now have from Observation~\ref{obs:concat} that $\{a,b\}\reach[1]{C'\cup Z_{i-1}}\{w,y_{i-1}\}$. Recalling that $y_{i-1}v_i\in E(G)$, we have that $\{w,y_{i-1}\}\step\{w,v_i\}$. Since $v_i\notin C'\cup Z_{i-1}$ and $C'\cup Z_{i-1}\subseteq Z_i$, we can now conclude that $\{a,b\}\reach[1]{Z_i}\{w,v_i\}$. As $w\in C'$ and $C'$ is adjacent to $v_{i-1}$, we have $w\in R_{ab}(v_{i-1})$. The proof is completed by setting $y_i=w$.
\end{proof}
\begin{corollary}\label{cor:abtouy}
Let $\{a,b\}$ be an independent set in $G$, $u\in V(G)\setminus (C_{ab}(a)\cup C_{ab}(b))$, and $(v_0,v_1,\ldots,v_k)$ a minimum length $\{a,b\}$-chain to $u$. Let $$X=\left\{\begin{array}{ll}C_{ab}(a)\cup C_{ab}(b)&\mbox{ if }k=0\\R_{ab}(v_{k-1})&\mbox{ otherwise}\end{array}\right.$$ and $Z=C_{ab}(a)\cup C_{ab}(b)\cup\{v_k\}\cup\bigcup_{i=1}^k R_{ab}(v_{i-1})$. Then there exists $y\in X$ such that $\{a,b\}\reach[1]{Z}\{v_k,y\}$ and $\{a,b\}\reach[1]{}\{u,y\}$.
\end{corollary}
\begin{proof}
 By Lemma~\ref{lem:reconfig}, there exists $y_k\in X$ such that $\{a,b\}\reach[1]{Z}\{v_k,y_k\}$. Let $y=y_k$. We immediately have $\{a,b\}\reach[1]{Z}\{v_k,y\}$. We shall be done if we prove that $\{a,b\}\reach[1]{}\{u,y\}$. If $u\in N(a)\cap N(b)$, then $v_k=u$ and we already have $\{a,b\}\reach[1]{}\{u,y\}$. So we assume that $u\notin N(a)\cap N(b)$. Then $v_k$ is adjacent to $C_{ab}(u)$. Observe that $C_{ab}(u)\neq C_{ab}(a)$ and $C_{ab}(u)\neq C_{ab}(b)$ as $u\in V(G)\setminus (C_{ab}(a)\cup C_{ab}(b))$. If $k=0$, then $y\in C_{ab}(a)\cup C_{ab}(b)$, and therefore $y$ is not adjacent to any vertex in $C_{ab}(u)$. On the other hand if $k\geq 1$, then since we have by Observation~\ref{obs:minimumchain}\ref{ucomponent} that $C_{ab}(u)$ is not adjacent to $v_{k-1}$, it follows that no vertex in $C_{ab}(u)$ is adjacent to a vertex in $X=R_{ab}(v_{k-1})$. So we again have that $y$ is not adjacent to any vertex in $C_{ab}(u)$. Then any arbitrarily chosen path between $v_k$ and $u$ in $G[\{v_k\}\cup C_{ab}(u)]$ has no neighbours of $y$ on it, and therefore we have by Observation~\ref{obs:easymove} that $\{v_k,y\}\reach[1]{C_{ab}(u)\cup\{v_k,y\}}\{u,y\}$.
 Since we have $C_{ab}(u)\cap Z=\emptyset$ by Observation~\ref{obs:minimumchain}\ref{ucomponent}, this together with the fact that $\{a,b\}\reach[1]{Z}\{v_k,y\}$ implies by Observation~\ref{obs:concat} that $\{a,b\}\reach[1]{}\{u,y\}$.
\end{proof}

\begin{lemma}\label{lem:mobility}
    Let $\{a,b\}$ be an independent set in $G$. A vertex $u\in V(G)\setminus (C_{ab}(a)\cup C_{ab}(b))$ is reachable (recall Definition~\ref{def:reachable}) from $\{a,b\}$ if and only if there exists an $\{a,b\}$-chain to $u$.
\end{lemma}
\begin{proof}
    ($\Leftarrow$) Suppose that there exists an $\{a,b\}$-chain to a vertex $u\in V(G)\setminus (C_{ab}(a)\cup C_{ab}(b))$. Let $(v_0,v_1,\ldots,v_k)$ be a minimum length $\{a,b\}$-chain to $u$. It now follows from Corollary~\ref{cor:abtouy} that $u$ is reachable from $\{a,b\}$.
    \medskip

    \noindent($\Rightarrow$) For a vertex $u\in V(G)\setminus (C_{ab}(a)\cup C_{ab}(b))$ that is reachable from $\{a,b\}$, let us define $\tau(u)$ to be the least integer $t$ such that there is a reconfiguration sequence of length $t$ from $\{a,b\}$ to an independent set containing $u$. For proving that for every vertex $u\in V(G)\setminus (C_{ab}(a)\cup C_{ab}(b))$ that is reachable from $\{a,b\}$, there is an $\{a,b\}$-chain to $u$, we shall use induction on $\tau(u)$. Note that $\tau(u)\geq 2$ as $u\notin C_{ab}(a)\cup C_{ab}(b)$ implies that $u\notin\{a,b\}$. As the base case, if $\tau(u)=2$, then we have either $\{a,b\}\step\{a,u\}$ or $\{a,b\}\step\{b,u\}$. This means that one of $\{a,b\}$ is adjacent to $u$ and the other is nonadjacent to $u$. Since $u\notin C_{ab}(a)\cup C_{ab}(b)$, the former implies that $u\in N(a)\cap N(b)$. The latter implies that one of $C_{ab}(a),C_{ab}(b)$ avoids $u$. We can conclude that $(u)$ is an $\{a,b\}$-chain to $u$, and we are done for the base case. So we shall assume that $\tau(u)>2$. Let $\mathcal{S}$ be a reconfiguration sequence of length $\tau(u)$ from $\{a,b\}$ to an independent set $I$ containing $u$. Let $I=\{u,w\}$. Let $I'$ be the independent set just before $I$ in the sequence $\mathcal{S}$. By our choice of $\mathcal{S}$, we have that $u\notin I'$, and since $I'\step I$, we can assume that $I'=\{w,y\}$ and $yu\in E(G)$. Since there is a reconfiguration sequence from $\{a,b\}$ to $I'$ of length $\tau(u)-1$, we have that both $w$ and $y$ are reachable from $u$, and in particular $\tau(w)<\tau(u)$ and $\tau(y)<\tau(u)$.
    
    Suppose that $u\notin N(a)\cap N(b)$. Then as $yu\in E(G)$ and $u\notin C_{ab}(a)\cup C_{ab}(b)$, we have that $y\notin C_{ab}(a)\cup C_{ab}(b)$. Since $\tau(y)<\tau(u)$, we can assume by the induction hypothesis that there is an $\{a,b\}$-chain $(v_0,v_1,\ldots,v_k)$ to $y$. As $yu\in E(G)$, we have that either $y\in N(a)\cap N(b)$ or $y\in C_{ab}(u)$. If $y\in N(a)\cap N(b)$, then $v_k=y$ and $C_{ab}(u)$ is adjacent to $v_k$. On the other hand, if $y\in C_{ab}(u)$, then $v_k$ is adjacent to $C_{ab}(y)=C_{ab}(u)$. In either case, we have that $(v_0,v_1,\ldots,v_k)$ is an $\{a,b\}$-chain to $u$.
    
    Next, suppose that $u\in N(a)\cap N(b)$. If $w\in C_{ab}(a)\cup C_{ab}(b)$, then as $uw\notin E(G)$ (as $\{u,w\}$ is an independent set), we have that $u$ is avoided by $C_{ab}(w)$. Then as $C_{ab}(w)=C_{ab}(a)$ or $C_{ab}(w)=C_{ab}(b)$, it follows that $(u)$ is an $\{a,b\}$-chain to $u$, and we are done. So we can assume that $w\notin C_{ab}(a)\cup C_{ab}(b)$. Since $\tau(w)<\tau(u)$, we have by the induction hypothesis that there is an $\{a,b\}$-chain $(v_0,v_1,\ldots,v_k)$ to $w$. If $w\notin N(a)\cap N(b)$, then we have that $v_k$ is adjacent to $C_{ab}(w)$ and that $u$ is avoided by $C_{ab}(w)$. On the other hand, if $w\in N(a)\cap N(b)$, then we have $v_k=w$ and therefore $v_ku\notin E(G)$. In either case, we have that $(v_0,v_1,\ldots,v_k,u)$ is an $\{a,b\}$-chain to $u$. This completes the proof.
\end{proof}

\begin{corollary}\label{cor:generalniv}
Let $\{a,b\}$ be an independent set in $G$. If $v \in V(G)$ is reachable from $\{a,b\}$, then there exists a vertex $u \in V(G)$ such that $\{a,b\}\reach[1]{}\{u,v\}$.
\end{corollary}
\begin{proof}
The result follows immediately from Corollary~\ref{cor:firstmove}, Corollary~\ref{cor:abtouy}, and Lemma~\ref{lem:mobility}.
\end{proof}

\begin{lemma}\label{lem:mobilityniv}
Let $\{a,b\}$ be an independent set and let $u\in V(G_{ab})\setminus (C_{ab}(a)\cup C_{ab}(b))$ be a vertex that is reachable from $\{a,b\}$. Then:
\vspace{-0.05in}
\begin{enumerate}
    \setlength{\itemsep}{0in}
    \renewcommand{\theenumi}{(\roman{enumi})}
    \renewcommand{\labelenumi}{\theenumi}
    \item\label{or} $\{a,b\}\reach[2]{}\{a,u\}$ or $\{a,b\}\reach[2]{}\{b,u\}$, and
    \item\label{and} if there exists a vertex in $N(a)\cap N(b)$ that is avoided by $C_{ab}(u)$, then $\{a,b\}\reach[2]{}\{a,u\}$ and $\{a,b\}\reach[2]{}\{b,u\}$.
\end{enumerate}
\end{lemma}
\begin{proof}
Since $u$ is reachable from $\{a,b\}$, we have from Lemma~\ref{lem:mobility} that there is an $\{a,b\}$-chain to $u$. Let $(v_0,v_1,\ldots,v_k)$ be a minimum length $\{a,b\}$-chain to $u$. Let $X=C_{ab}(a)\cup C_{ab}(b)$ if $k=0$ and $X=R_{ab}(v_{k-1})$ otherwise. Further let $Z=C_{ab}(a)\cup C_{ab}(b)\cup\{v_k\}\cup\bigcup_{i=1}^k R_{ab}(v_{i-1})$. From Corollary~\ref{cor:abtouy}, we have that there exists $y\in X$ such that $\{a,b\}\reach[1]{Z}\{v_k,y\}$ and $\{a,b\}\reach[1]{}\{u,y\}$.

Let us first consider the case when $y\in C_{ab}(a)\cup C_{ab}(b)$. Note that $k=0$ in this case.
Let $x\in\{a,b\}$ such that $y\in C_{ab}(x)$. Observe that $C_{ab}(u)\neq C_{ab}(x)$ as $u\in V(G_{ab})\setminus (C_{ab}(a)\cup C_{ab}(b))$. So the path $P$ in $C_{ab}(x)$ between $x$ and $y$ is a path in $G$ between $x$ and $y$ that does not contain any neighbour of $u$. We then have from Observation~\ref{obs:easymove} that $\{u,y\}\reach[1]{}\{u,x\}$. As $\{a,b\}\reach[1]{}\{u,y\}$, it follows from Observation~\ref{obs:concat} that $\{a,b\}\reach[2]{}\{u,x\}$. Thus~\ref{or} holds for the case when $y\in C_{ab}(a)\cup C_{ab}(b)$. Now suppose that there exists a vertex $w\in N(a)\cap N(b)$ that is avoided by $C_{ab}(u)$. Then $w$ is not adjacent to some vertex $z\in C_{ab}(u)$. Recall that $k=0$ and therefore $\{a,b\}\reach[1]{Z}\{v_0,y\}$. Since $v_0$ is adjacent to $C_{ab}(u)=C_{ab}(z)$, there is a path in $G$ between $v_0$ and $z$ that contains no neighbours of $y$. It follows from Observation~\ref{obs:easymove} that $\{v_0,y\}\reach[1]{C_{ab}(u)\cup\{v_0,y\}}\{z,y\}$. As $k=0$, we have that $Z=C_{ab}(a)\cup C_{ab}(b)\cup\{v_0\}$, which implies that $Z\cap (C_{ab}(u)\cup\{v_0,y\})=\{v_0,y\}$. Since $\{a,b\}\reach[1]{Z}\{v_0,y\}$, it now follows from Observation~\ref{obs:concat} that $\{a,b\}\reach[1]{}\{z,y\}$.
Let $Q$ be the path in $C_{ab}(u)$ between $u$ and $z$. Also, let $x'$ denote the vertex in $\{a,b\}\setminus\{x\}$, i.e. $\{x,x'\}=\{a,b\}$. Since the path $yPxwx'$ does not contain any neighbour of $z$ (note that $C_{ab}(u)\neq C_{ab}(x)$ and $C_{ab}(u)\neq C_{ab}(x')$), we have from Observation~\ref{obs:easymove} that $\{z,y\}\reach[1]{C_{ab}(a)\cup C_{ab}(b)\cup\{w,z\}}\{z,x'\}$. Again, the path $Q$ does not contain any neighbour of $x'$, and we have from Observation~\ref{obs:easymove} that $\{z,x'\}\reach[1]{C_{ab}(u)\cup\{x'\}}\{u,x'\}$. It now follows from Observation~\ref{obs:concat} that $\{z,y\}\reach[1]{}\{u,x'\}$. As $\{a,b\}\reach[1]{}\{z,y\}$, Observation~\ref{obs:concat} further gives us $\{a,b\}\reach[2]{}\{u,x'\}$, and we are done. This proves~\ref{and} for the case when $y\in C_{ab}(a)\cup C_{ab}(b)$.

So let us assume that $y\notin C_{ab}(a)\cup C_{ab}(b)$. Then we have that $k>0$ and $y\in R_{ab}(v_{k-1})\setminus (C_{ab}(a)\cup C_{ab}(b))$. If $y\in N(a)\cap N(b)$, then $\{u,y\}\step\{u,a\}$ and $\{u,y\}\step\{u,b\}$, which when combined with the fact that $\{a,b\}\reach[1]{}\{u,y\}$ gives $\{a,b\}\reach[2]{}\{u,a\}$ and $\{a,b\}\reach[2]{}\{u,b\}$. This proves both~\ref{or} and~\ref{and}. So we assume that $y\notin N(a)\cap N(b)$. As $y\in R_{ab}(v_{k-1})$, we have that $C_{ab}(y)$ is adjacent to $v_{k-1}$. Since we have by Observation~\ref{obs:minimumchain}\ref{ucomponent} that $v_{k-1}$ is not adjacent to any vertex in $C_{ab}(u)$, we have $C_{ab}(y)\neq C_{ab}(u)$. This further implies that no vertex in an arbitrarily chosen path $P$ between $y$ and $v_{k-1}$ in $G[C_{ab}(y)\cup\{v_{k-1}\}]$ is adjacent to $u$. Thus the paths $yPv_{k-1}a$ and $yPv_{k-1}b$ contain no neighbours of $u$, and therefore we have by Observation~\ref{obs:easymove} that $\{u,y\}\reach[1]{}\{u,a\}$ and $\{u,y\}\reach[1]{}\{u,b\}$. As before, this proves both~\ref{or} and~\ref{and}.
\end{proof}
\begin{theorem}\label{thm:final}
Let $\{a,b\}$ and $\{c,d\}$ be independent sets in a graph $G$ such that
$\{a,b\}\reach{}\{c,d\}$. Then $\{a,b\}\reach[4]{}\{c,d\}$.
\end{theorem}
\begin{proof}
As $a$ is reachable from $\{c,d\}$, there exists $u\in V(G)$ such that $\{c,d\}\reach[1]{}\{a,u\}$ by Corollary~\ref{cor:generalniv}. We will be done if we show that $\{a,u\}\reach[3]{}\{a,b\}$. Note that since $\{a,b\}\reach{}\{c,d\}$, we already know that $\{a,u\}\reach{}\{a,b\}$. If $u\in C_a(b)$, then there is a path $P$ in $G$ from $u$ to $b$ that does not contain any neighbour of $a$. Then we have by Observation~\ref{obs:easymove} that $\{a,u\}\reach[1]{}\{a,b\}$, and we are done. So let us assume that $u\notin C_a(b)$.

Suppose now that $u\in C_b(a)$. If $b\in C_u(a)$, then we have by Observation~\ref{obs:easymove} that $\{a,u\}\reach[1]{}\{b,u\}\reach[1]{}\{a,b\}$, which implies that $\{a,u\}\reach[2]{}\{a,b\}$. So we can assume that $b\notin C_u(a)$. Recalling that $b\notin C_a(u)$ (as $u\notin C_a(b)$), we have from Lemma~\ref{lem:firstmove} that $b\notin C_{au}(a)\cup C_{au}(u)$. Since $b$ is reachable from $\{a,u\}$ (as we have $\{a,u\}\reach{}\{c,d\}$ and $b$ is reachable from $\{c,d\}$), we have from Lemma~\ref{lem:mobilityniv}\ref{or} that $\{a,u\}\reach[2]{}\{a,b\}$ or $\{a,u\}\reach[2]{}\{b,u\}$ (note that $b\notin N(a)\cap N(u)$). Since we are done in the former case, we shall assume the latter case, i.e. $\{a,u\}\reach[2]{} \{b,u\}$. As $u\in C_b(a)$, there is a path $P$ between $u$ and $a$ that does not contain any neighbour of $b$. We thus have by Observation~\ref{obs:easymove} that $\{b,u\}\reach[1]{}\{a,b\}$. As a result, in this case we have $\{a,u\}\reach[3]{}\{a,b\}$ and therefore we are done.

So, we may assume from now on that $u\notin C_b(a)$. Since we also have $u\notin C_a(b)$, it follows from Lemma~\ref{lem:firstmove} that $u\notin C_{ab}(a)\cup C_{ab}(b)$. Now, as $\{a,b\}\reach{}\{c,d\}$ and $u$ is reachable from $\{c,d\}$, it follows that $u$ is reachable from $\{a,b\}$. Clearly, $u\notin N(a)\cap N(b)$. Hence, by Lemma~\ref{lem:mobilityniv}\ref{or}, we have either $\{a,b\}\reach[2]{}\{a,u\}$ or $\{a,b\}\reach[2]{}\{b,u\}$. We are done in the former case. So we can assume the latter case, i.e. $\{a,b\}\reach[2]{}\{b,u\}$.

Suppose that $a\in C_u(b)$. Then there exists a path $P$ from $b$ to $a$ that does not contain any neighbour of $u$, which implies by Observation~\ref{obs:easymove} that $\{b,u\}\reach[1]{}\{a,u\}$. Thus $\{a,b\}\reach[3]{}\{a,u\}$, and we are done.

So let us assume that $a\notin C_u(b)$. Thus we are now in the case where
$a\notin C_u(b)$, $b\notin C_a(u)$, and $a\notin C_b(u)$. It is easy to see that this is possible only if
$N(a)\cap N(b)=N(a)\cap N(u)=N(b)\cap N(u)=N(a)\cap N(b)\cap N(u)$,
and the removal of this common neighbourhood separates $a$, $b$, and $u$ into three distinct connected components.

If $C_{ab}(u)$ avoids any vertex of $N(a)\cap N(b)$, then by Lemma~\ref{lem:mobilityniv}\ref{and}, we have $\{a,u\}\reach[2]{}\{a,b\}$, and we are done. Similarly, if $C_{au}(b)=C_{ab}(b)$ avoids any vertex of $N(a)\cap N(u)$, we have by Lemma~\ref{lem:mobilityniv} that $\{a,u\}\reach[2]{}\{a,b\}$. Therefore, we may assume that both $C_{ab}(u)$ and $C_{ab}(b)$ dominate every vertex of $N(a)\cap N(b)$. If $C_{ab}(a)$ also dominates every vertex of $N(a)\cap N(b)$, then Corollary~\ref{cor:impossible} implies that $\{a,b\}\nreach{}\{a,u\}$, contradicting the fact that $\{a,b\}\reach{}\{a,u\}$. Hence, there exists a vertex in $w\in N(a)\cap N(b)$ that is avoided by $C_{ab}(a)$. Then there exists $y\in C_{ab}(a)$ such that $yw\notin E(G)$. Let $P$ be a path in $C_{ab}(a)$ between $a$ and $y$. Clearly, $P$ contains no neighbours of $b$ or $u$. Since $C_{ab}(u)$ and $C_{ab}(b)$ dominate $N(a)\cap N(b)$, we have that $b,u\in N(w)$. We then have a path $aPy$ that contains no neighbours of $u$ or $b$, and a path $bwu$ that contains no neighbours of $y$. From two applications of Observation~\ref{obs:easymove}, we get $\{a,u\}\reach[1]{C_{ab}(a)\cup\{u\}}\{y,u\}$ and $\{y,u\}\reach[1]{\{u,b,w,y\}}\{y,b\}$. Combining using Observation~\ref{obs:concat}, we have $\{a,u\}\reach[1]{}\{y,b\}$. We can apply Observation~\ref{obs:easymove} again on the path $aPy$ to derive $\{y,b\}\reach[1]{}\{a,b\}$. Combining this with $\{a,u\}\reach[1]{}\{y,b\}$, we obtain $\{a,u\}\reach[2]{}\{a,b\}$, thereby completing the proof.
\end{proof}
\bigskip

\noindent\textit{Proof of Theorem~\ref{thm:main}:} It follows from Theorem~\ref{thm:final} and Observation~\ref{obs:length} that if $I,I'$ are two independent sets in a graph $G$ on $n$ vertices such that $|I|=|I'|=2$ and $I\reach{} I'$, then there is a reconfiguration sequence of length at most $4n$ between $I$ and $I'$. This shows that $\mu_2(n)\leq 4n$.\hfill\qed
\bibliography{main}
\end{document}